\documentclass{article}
\usepackage[utf8]{inputenc}
\usepackage{amsmath, amsthm, amssymb, enumitem, tikz, scalerel}

\let\OLDthebibliography\thebibliography
\renewcommand\thebibliography[1]{
  \OLDthebibliography{#1}
  \setlength{\parskip}{0pt}
  \setlength{\itemsep}{1pt plus 0.3ex}
}

\usepackage[letterpaper, portrait, margin = 1.5in]{geometry}
\usepackage[hyperfootnotes=false]{hyperref}
\usepackage[capitalize,noabbrev]{cleveref}
\allowdisplaybreaks

\newtheorem{theorem}{Theorem}[section]
\newtheorem{corollary}[theorem]{Corollary}
\newtheorem{proposition}[theorem]{Proposition}
\newtheorem{lemma}[theorem]{Lemma}
\newtheorem{conjecture}[theorem]{Conjecture}
\theoremstyle{definition}
\newtheorem{definition}[theorem]{Definition}
\newtheorem{example}[theorem]{Example}
\newtheorem{question}[theorem]{Question}

\numberwithin{equation}{section}

\newcommand*{\R}{\mathbb{R}}
\newcommand*{\Z}{\mathbb{Z}}
\newcommand*{\N}{\mathbb{N}}
\newcommand*{\h}{H_{\scaleobj{0.9}{\square}}}
\DeclareMathOperator{\smallmod}{mod}
\DeclareMathOperator{\conv}{conv}
\DeclareMathOperator{\bx}{\mathsf{box}}

\makeatletter
\def\blankfootnote{\xdef\@thefnmark{}\@footnotetext}
\makeatother

\newcommand{\inserttitle}{discrete quantitative helly-type theorems with boxes} 

\title{\vspace{-1.5em} \inserttitle}
\author{Travis Dillon}
\date{\today}

\usepackage{fancyhdr}
\fancypagestyle{firstpage}
{
   \fancyhf{}
}

\begin{document}


\thispagestyle{firstpage}

\makeatletter

\null
\vspace{-2.5em}

\noindent
    \vbox{%
        \hsize\textwidth
        \linewidth\hsize
        {
          \hrule height 2\p@
          \vskip 0.25in
          \vskip -\parskip%
        }
        \centering
        {\LARGE\sc discrete quantitative helly-type\\ theorems with boxes\par}
        {
          \vskip 0.26in
          \vskip -\parskip
          \hrule height 2\p@
          \vskip 0.10in%
        }
    }
\makeatother
\begin{center}\large
    Travis Dillon\\[0.4em]
    
    \today\\[1.7em]
\end{center}

\begin{abstract}
\noindent
Research on Helly-type theorems in combinatorial convex geometry has produced volumetric versions of Helly's theorem using witness sets and quantitative extensions of Doignon's theorem. This paper combines these philosophies and presents quantitative Helly-type theorems for the integer lattice with axis-parallel boxes as witness sets. Our main result shows that, while quantitative Helly numbers for the integer lattice grow polynomially in each fixed dimension, their variants with boxes as witness sets are uniformly bounded. We prove several colorful and fractional variations on this theorem. We also prove that the Helly number for $A \times A \subseteq \mathbb{R}^2$ need not be finite even when $A \subseteq \mathbb{Z}$ is a syndetic set.\blankfootnote{Lawrence University, email:\,\textit{travis.a.dillon\raisebox{-0.07em}{\fontfamily{qtm}\selectfont @}\hspace{-0.08em}lawrence.edu}}

\end{abstract}

\section{Introduction}

A significant portion of combinatorial geometry comprises the study of intersection patterns of convex sets. A fundamental result of this type is Helly's theorem, which states that \textit{if the intersection of any $d+1$ or fewer members of a finite family of convex sets in $\R^d$ is nonempty, then the entire family has nonempty intersection} \cite{Helly}. Doignon \cite{Doignon} obtained a similar theorem for the integer lattice (which was also independently discovered by Bell \cite{bell-doignon} and Scarf \cite{scarf-doignon}):

\begin{theorem}[Doignon, 1973]
If the intersection of any $2^d$ or fewer members of a finite family of convex sets in $\R^d$ contains a point in $\Z^d$, then the intersection of the entire family contains a point in $\Z^d$.
\end{theorem}

We refer to points in $\Z^d$ as \textit{lattice points} for brevity. Recently, there has been significant interest in extending Doignon's and Helly's theorems; the theory has grown quickly, and Amenta, De Loera, and Sober\'on \cite{Soberon-Helly-survey} as well as Holmsen and Wenger \cite{book-Helly-survey} survey many results of this type. One way to extend Doignon's theorem is to replace $\Z^d$ by an arbitrary set $S \subseteq \R^d$; another is to ask for not just one intersection point, but many. These two approaches lead to the following definition.

\begin{definition}\label{def:Helly-number}
Given a set $S \subseteq \R^d$, we let $H(S,n)$ denote the least integer, if it exists, such that the following statement is true for every finite family $\mathcal F$ of convex sets in $\R^d$: If the intersection of any $H(S,n)$ or fewer members of $\mathcal F$ contains at least $n$ points of $S$, then $\bigcap \mathcal F$ contains at least $n$ points of $S$. If no such number exists, we write $H(S,n) = \infty$. We abbreviate $H(S,1)$ by $H(S)$.
\end{definition}

\begin{example}
Helly's and Doignon's theorems state that $H(\R^d) \leq d+1$ and $H(\Z^d) \leq 2^d$, respectively. On the other hand, any $d$ facets of a $d$-dimensional simplex intersect at a vertex, while the intersection of all $d+1$ facets is empty; this shows that $H(\R^d) > d$. To show that $H(\Z^d) = 2^d$, let $Q = \{0,1\}^d$ denote the set of vertices of the unit cube in $\R^d$. The intersection of any $2^d-1$ sets in $\{ \conv(Q \setminus x) : x \in Q\}$ contains a lattice point, but the intersection of all $2^d$ contains none.
\end{example}

For $n \geq 1$, we call $H(S,n)$ a \textit{quantitative Helly number for $S$}. Averkov, Gonz\'alez Merino, Paschke, Schymura, and Weltge \cite{tight-bounds-discrete-helly} evaluated $H(S,n)$ for discrete sets $S$ in terms of polytopes with vertices in $S$. Many results in this area focus on quantitative Helly numbers for $\Z^d$. Aliev, Bassett, De Loera, and Louveaux \cite{quantitative-Doignon-Bell-Scarf}, the first to study quantitative Helly-type theorems for the integer lattice specifically, proved that $H(\Z^d,n) \leq (2^d-2)\lceil 2n/3\rceil + 2$. This was improved in \cite{tight-bounds-discrete-helly} to $H(\Z^d,n) = \Theta(n^{\frac{d-1}{d+1}})$ for every fixed $d \geq 1$. Other authors have investigated discrete sets with different structure: De Loera, La Haye, Oliveros, and Rold\'an-Pensado \cite{Helly-algebraic-subsets} and De Loera, La Haye, Rolnick, and Sober\'on \cite{quantitative-tverberg-survey} bound $H(S,n)$ above when $S$ is a difference of lattices, and Garber \cite{Garber-crystals} studied $H(S)$ when $S$ is a crystal or a cut-and-project set.

Further Helly-type theorems can be obtained by restricting the type of sets in the family under consideration. Intersection properties of axis-parallel boxes are particularly well-studied (see \cite{box-intersection} and the references therein). For example, a family of axis-parallel boxes in $\R^d$ has nonempty intersection if and only if every pair of boxes has nonempty intersection.

In another direction, Sarkar, Xue, and Sober\'on \cite{REU2019} proved that \textit{if the intersection of every $2d$ or fewer members of a finite family of convex sets in $\R^d$ contains an axis-parallel box with volume one, then the intersection of the entire family contains an axis-parallel box with volume one}. In this theorem, the axis-parallel box acts as a ``witness set'' to the size of the intersection. Theorems with witness sets can differ strikingly from the corresponding results without them: There is no exact Helly-type theorem for volume without witness sets, for example \cite{quant-continuous-param}. Sarkar, Xue, and Sober\'on obtained many results in this vein; currently, however, there are no discrete quantitative Helly theorems within this framework.

This paper's results can be broadly divided into two categories. First, we prove several discrete quantitative Helly-type theorems with axis-parallel boxes as witness sets. In analogy to \cref{def:Helly-number} and the work in \cite{REU2019}, we define the following modified Helly number.

\begin{definition}\label{def:box-Helly-number}
Given a set $S \subseteq \R^d$, we let $\h(S,n)$ denote the least integer, if it exists, such that the following statement is true for every finite family $\mathcal F$ of convex sets in $\R^d$: If the intersection of any $\h(S,n)$ or fewer members of $\mathcal F$ contains an axis-parallel box with at least $n$ points of $S$, then $\bigcap \mathcal F$ contains an axis-parallel box with at least $n$ points of $S$. If no such number exists, we write $\h(S,n) = \infty$.
\end{definition}

For the remainder of the paper, we write ``box'' to mean ``axis-parallel box''\!\!\!. It is not immediately clear how $H(S,n)$ and $\h(S,n)$ relate; while Helly-type theorems for boxes have a stronger conclusion---not only does $\bigcap \mathcal F$ contain $n$ points, but these points are themselves contained in a box---they also have a similarly strengthened premise. Our first main result shows that $\h(\Z^d,n)$ is uniformly bounded for every fixed $d \geq 1$.

\begin{theorem}\label{thm:box-helly-integer-lattice}
$\h(\Z^d, n) \leq 2^{2d-1}$ for every $n \geq 1$ and $d\geq 1$.
\end{theorem}

The behavior of $\h(\Z^d,n)$ is therefore quite different from that of the usual quantitative Helly number $H(\Z^d,n)$, which grows on the order of $n^{\frac{d-1}{d+1}}$. A second result shows that we can still guarantee a relatively large global intersection with a much weaker local intersection condition. (1-thickness is a mild technical condition introduced in \cref{sec:box-Helly}.)

\begin{theorem}\label{thm:lower-Helly-number}
Fix a positive integer $d \geq 1$ and let $\mathcal F$ be a finite family of convex sets in $\R^d$. If the intersection of every $2d$ or fewer members of $\mathcal F$ contains a $1$-thick box with at least $n$ lattice points, then $\bigcap \mathcal F$ contains a box with at least $n/3^{d-1}$ lattice points.
\end{theorem}

In several aspects, \cref{thm:lower-Helly-number} cannot be improved: \cref{ex:exponential-loss-lower-Helly-number} shows that reducing $2^{2d-1}$ in \cref{thm:box-helly-integer-lattice} to any number smaller than $2^d$ necessarily entails an exponential loss of lattice points in the intersection, and \cref{ex:impossible-fractional} shows that the number $2d$ cannot be reduced while maintaining any positive fraction of points in the intersection. We also obtain colorful versions of theorems \ref{thm:box-helly-integer-lattice} and \ref{thm:lower-Helly-number} (see \cref{thm:colorful-box-helly-integer} and the discussion following \cref{thm:colorful-mixed-integer}). In addition, we prove an analogue of \cref{thm:box-helly-integer-lattice} for the more general collection of \textit{periodic product sets} (defined in \cref{sec:box-Helly}), along with several fractional theorems.

The second collection of results establishes infinite Helly numbers for certain collections of product sets. Given a set $A \subseteq \R$, we let $A^d$ denote the $d$-fold Cartesian product of $A$. Our work is inspired by the following conjecture of De Loera, La Haye, Oliveros, and Rold\'an-Pensado \cite{Helly-algebraic-subsets}.

\begin{conjecture}
If $\mathcal P$ represents the set of prime numbers, then $H(\mathcal P^2) = \infty$.
\end{conjecture}

\cref{thm:product-sets-lower-bound} gives a general method to establish a lower bound for $H(A^2)$ whenever $A$ is a discrete subset of $\R$. We derive several corollaries, including that $H(A^d) = \infty$ for every $d\geq 2$ if $A = \{n^k : n \in \N\}$. It may not be too surprising that $H(A^d)$ is infinite, since $A$ is sparse in $\N$, that is, that $\liminf_{n\to\infty} \big\lvert A \cap [1,n]\big\rvert/n = 0$. However, we also prove that infinite Helly number is not solely attributable to sparseness.

\begin{theorem}\label{thm:syndetic-set-infinite-helly}
There is a set $A \subseteq \Z$ whose consecutive elements differ by at most $2$ such that $H(A^d) = \infty$ for every $d \geq 2$.
\end{theorem}

Sets with bounded gaps are often referred to as \textit{syndetic}. Requiring bounded gaps is a strictly stronger condition than positive density. As a corollary, \cref{thm:syndetic-set-infinite-helly} provides another proof that Delone sets need not have finite Helly number. (A set $S \subseteq \R^d$ is called a \textit{Delone set} if there exist $r,R \geq 0$ so that the minimum distance between any two points in $S$ is bounded below by $r$ and every ball of radius $R$ contains at least one point of $S$.)

Several new questions are posed throughout the paper.

\section{Quantitative Helly-type theorems with boxes}\label{sec:box-Helly}

\subsection{Exact Helly-type theorems}

Throughout this section, $e_1,\dots,e_d$ denote the standard basis vectors in $\R^d$ and $\pi_i\colon \R^d \to \R^{d-1}$ denotes the orthogonal projection along $e_i$. We begin with a parametrization of boxes in $\R^d$.

\begin{definition}
Given two vectors $x,y \in \R^d$, we write $y \geq x$ when the inequality is true in each coordinate. If $y\geq x$, we denote the axis-parallel box $\prod_{i=1}^d [x_i,y_i]$ by $\bx(x,y)$. Given two sets $S,T \subseteq \R^d$, we say that $\bx(x,y)$ is \textit{parametrized in $S\times T$} if $(x,y) \in S \times T$. For each set $F \subseteq \R^d$, we define $B_F = \{(x,y) \in \R^d \times \R^d : y \geq x \text{ and } \bx(x,y) \subseteq F\}$.
\end{definition}

If $F$ is convex, then $B_F$ is convex, since the convex combination of any two boxes is determined by the convex combination of their vertices. We can interpret a function $f\colon \R^d \times \R^d \to \R$ as a function on axis-parallel boxes by setting $f(\bx(x,y)) = f(x,y)$. We call such a function a \textit{box-weighting function in $\R^d$}. Choices for $f$ might be the number of lattice points contained in, volume of, or surface area of $\bx(x,y)$.

\begin{proposition}\label{thm:basic-box-Helly}
Let $\mathcal F$ be a finite family of convex sets in $\R^d$ and $S,T \subseteq \R^d$. Fix a box-weighting function $f$ and some $w \in \R$. If the intersection of any $H\big((S\times T )\cap f^{-1}([w,\infty))\big)$ or fewer members of $\mathcal F$ contains a box parametrized in $S \times T$ with weight at least $w$, then $\bigcap \mathcal F$ contains a box parametrized in $S \times T$ with weight at least $w$, as well.
\end{proposition}
\begin{proof}
By assumption, any $H\big((S\times T) \cap f^{-1}([w,\infty))\big)$ members of $\{B_F\}_{F \in \mathcal F}$ contains a point in $(S \times T) \cap f^{-1}([w,\infty))$, so there exists a point $(a,b) \in \bigcap_{F \in \mathcal F} B_F \cap (S \times T) \cap f^{-1}([w,\infty))$. Then $\bx(a,b)$ is parametrized in $S \times T$, contained in $\bigcap \mathcal F$, and has weight at least $w$.
\end{proof}

We focus on functions $f$ where $H\big((S \times T) \cap f^{-1}([w,\infty))\big)$ has an upper bound independent of $w$. For certain box-weighting functions, we can improve the bound in \cref{thm:basic-box-Helly}.

\begin{definition}
Let $X \subseteq \R^d$ be a convex set. A function $f\colon X \to \R$ is called \textit{min-concave} if $f\big(tx+(1-t)y\big) \geq \min\{f(x),f(y)\}$ for every $x,y \in X$ and $t \in [0,1]$.
\end{definition}

Min-concave functions are also often called \textit{quasiconcave}. A function $f\colon X \to \R$ is min-concave if and only if $\{ x \in X : f(x) \geq w\}$ is convex for every $w \in \R$.

\begin{proposition}\label{thm:box-weighting-min-concave}
Let $\mathcal F$ be a finite family of convex sets in $\R^d$ and $S,T \subseteq \R^d$. Fix a min-concave box-weighting function $f$ in $\R^d$ and some $w \in \R$. If the intersection of any $\min_i H(\pi_i(S \times T))$ or fewer members of $\mathcal F$ contains a box parametrized in $S \times T$ with weight at least $w$, then $\bigcap \mathcal F$ contains a box parametrized in $S \times T$ with weight at least $w$, as well.
\end{proposition}
\begin{proof}
Let $k \in [2d]$ be an index that minimizes $H(\pi_k(S\times T))$ and set
\[  V = \{(x,y) \in \R^d \times \R^d : y \geq x \text{ and } f(x,y) \geq w\}. \]
By assumption, the intersection of any $H(\pi_k(S\times T))$ members of $\{\pi_k(B_F \cap V)\}_{F \in \mathcal F}$ contains a point in $\pi_k(S\times T)$, and each set $B_F \cap V$ is convex; therefore $\bigcap_{F \in \mathcal F} \pi_k(B_F \cap V)$ contains a point $a \in \pi_k(S\times T)$. For each $F \in \mathcal F$, there is a point $a_F \in B_F \cap V \cap (S \times T)$ that projects to $a$. The sets $\bx(a_F)$ are ordered linearly by containment, so choose a $G \in \mathcal F$ so that $\bx(a_G) \subseteq \bx(a_F)$ for every $F \in \mathcal F$. Then $\bx(a_G) \subseteq \bigcap \mathcal F$ and $f(a_G) \geq w$.
\end{proof}

We can substitute any Helly-type theorem for $S\times T$ or $\pi_k(S\times T)$ in the proof of \cref{thm:basic-box-Helly} or \cref{thm:box-weighting-min-concave}, respectively, to obtain the corresponding result with boxes. For example, we could substitute a so-called \textit{colorful} Helly-type theorem, the first of which was proved by Lov\'asz and appeared in \cite{Colorful-Helly}.

\begin{theorem}[Colorful Helly theorem]
Let $\mathcal F_1, \dots,\mathcal F_{d+1}$ be nonempty finite families of convex sets in $\R^d$. If $\bigcap_{i=1}^{d+1} F_i$ is nonempty for every choice of sets $F_i \in \mathcal F_i$, then there is an index $k \in [d+1]$ so that $\bigcap \mathcal F_k$ is nonempty.
\end{theorem}

The term ``colorful'' comes from thinking of each finite family as a color class of sets. De Loera, La Haye, Oliveros, and Rold\'an-Pensado \cite{Helly-algebraic-subsets} proved a similar theorem for the integer lattice.

\begin{theorem}[Colorful Doignon theorem]\label{thm:colorful-doignon}
Let $\mathcal F_1, \dots,\mathcal F_{2^d}$ be nonempty finite families of convex sets in $\R^d$. If $\bigcap_{i=1}^{2^d} F_i$ contains a lattice point for every choice of sets $F_i \in \mathcal F_i$, then there is an index $k \in [2^d]$ so that $\bigcap \mathcal F_k$ contains a lattice point.
\end{theorem}

We generalize this to a colorful mixed-integer Helly-type theorem in \cref{thm:colorful-mixed-integer}. Using \cref{thm:colorful-doignon}, we can prove a colorful version of \cref{thm:box-helly-integer-lattice}.

\begin{theorem}\label{thm:colorful-box-helly-integer}
Set $h = 2^{2d-1}$ and let $\mathcal F_1, \dots,\mathcal F_h$ be nonempty finite families of convex sets in $\R^d$. If $\bigcap_{i=1}^h F_i$ contains a box with at least $n$ lattice points for every choice of sets $F_i \in \mathcal F_i$, then there is an index $k \in [h]$ so that $\bigcap \mathcal F_k$ contains a box with at least $n$ lattice points.
\end{theorem}
\begin{proof}
Set $f(x,y) := \prod_{i=1}^d (y_i - x_i + 1)$. If $(x,y) \in \Z^d \times \Z^d$, then the number of lattice points in $\bx(x,y)$ is exactly $f(x,y)$. In any case, the set $\bx(x,y)$ contains the same lattice points as $\bx(\lceil x\rceil, \lfloor y \rfloor)$, where the ceiling and floor operators are applied coordinate-wise. By assumption then, $\bigcap_{i=1}^h B_{F_i}$ contains a point $(a,b) \in \Z^d \times \Z^d$ with $f(a,b) \geq n$ whenever $F_i \in \mathcal F_i$ for every $i \in [h]$. The function $\log f$ is concave, so $f$ is logconcave and therefore min-concave. Noting that $H(\pi_i(\Z^{2d})) = H(\Z^{2d-1}) = 2^{2d-1}$ for every $i \in [2d]$, we may apply the colorful Doignon version of \cref{thm:box-weighting-min-concave} with $S = T = \Z^d$ and $w = n$ to obtain an index $k \in [h]$ so that $\bigcap_{F \in \mathcal F_k} B_F$ contains a lattice point $(a,b)$. Then $\bx(a,b) \subseteq \bigcap\mathcal F_k$ and $f(a,b) \geq n$, meaning $\bx(a,b)$ contains at least $n$ lattice points.
\end{proof}

\cref{thm:box-helly-integer-lattice} follows as a special case.

\begin{proof}[Proof of \cref{thm:box-helly-integer-lattice}]
Take $\mathcal F_i = \mathcal F$ for every $i \in [2^{2d-1}]$ and apply \cref{thm:colorful-box-helly-integer}
\end{proof}

It is unclear whether the number $2^{2d-1}$ in theorems \ref{thm:box-helly-integer-lattice} and \ref{thm:colorful-box-helly-integer} is optimal. The sharpness of Doignon's theorem shows that it cannot be replaced by any number less than $2^d$. \cref{ex:exponential-loss-lower-Helly-number} provides another illustration of this, while \cref{ex:lower-bound-box-d=2} gives a slightly larger lower bound of 6 in the plane.\newline

\noindent
\begin{minipage}{0.65\textwidth}
    \begin{example}\label{ex:lower-bound-box-d=2}
        To each line in Figure 1, associate the closed half-plane bounded by that line that contains the central lattice point. The intersection of any five of these half-planes contains a box with four vertices (two of which are degenerate boxes, consisting of four colinear points). On the other hand, any box in the intersection of all six half-planes contains at most three points.
    \end{example}
\end{minipage}
\hfill
\begin{minipage}[c]{0.3\textwidth}
    \centering
    \begin{tikzpicture}[scale=0.5]
        \foreach \x in {-2,...,2}
            \node[draw,shape=circle,fill=black,inner sep=0mm,minimum size=1.7mm] at (\x,0) {};
        \foreach \x in {-1,0,1}{
            \node[draw,shape=circle,fill=black,inner sep=0mm,minimum size=1.7mm] at (\x,-1) {};
            \node[draw,shape=circle,fill=black,inner sep=0mm,minimum size=1.7mm] at (\x,1) {};}
        \draw[thick] (-1,-1.5) -- (-1,1.5)  (1,-1.5) -- (1,1.5);
        \draw[thick] (1,-1.5) -- (-2,0) -- (1,1.5)  (-1,1.5) -- (2,0) -- (-1,-1.5);
    \end{tikzpicture}\\[0.5ex]
    Figure 1
 \end{minipage}
 \hspace{0.02\textwidth}
 
\begin{question}
What is the least integer that can replace $2^{2d-1}$ in \cref{thm:box-weighting-min-concave}?
\end{question}

We now prove a version of \cref{thm:box-helly-integer-lattice} for the more general class of periodic product sets, albeit with a slightly weaker constant.

\begin{definition}
We call a discrete set $A \subseteq \R$ \textit{periodic} if there exists some $p \in \R\setminus\{0\}$ so that $x \in A$ if and only if $x + p \in A$. A discrete set $Q \subseteq \R^d$ is called a \textit{periodic product set} if it can be factored as $A_1 \times \cdots \times A_d$, where $A_1, \dots, A_d \subseteq \R$ are periodic. If $A_i$ has least period $p_i$ for each $i \in [d]$, we say that $Q$ has \textit{period} $(p_1,\dots,p_d) \in \R^d$ and call $\prod_{i=1}^d [0,p_i)$ the \textit{fundamental box} of $Q$. We denote by $m_i(Q)$ the number $\lvert A_i \cap [0,p_i) \rvert$ and by $\rho(Q)$ the number of points of $Q$ in its fundamental box, namely $\rho(Q) = \prod_{i=1}^d m_i(Q)$.
\end{definition}

A modification to the proof of Theorem 4.7 in \cite{sublinear-bounds-quantitative-doignon} shows that $H(Q,n) = \Omega_{Q}(n^{\frac{d-1}{d+1}})$ for every $d$-dimensional periodic product set $Q$. In contrast, $\h(Q,n)$ is uniformly bounded for every fixed $Q$. To prove this, we use a result of Hoffman. Let $S \subseteq \R^d$. A polygon whose vertices lie in $S$ is called \textit{empty} if its interior contains no point of $S$. A subset $X \subseteq S$ is called \textit{intersect-empty} if $\big(\bigcap_{x \in X} \conv(X \setminus x)\big) \cap S = \emptyset$.

\begin{proposition}[Hoffman \cite{binding-constraints}, 1979]\label{thm:empty-polytopes}
If $S \subseteq \R^d$, then $H(S)$ is equal to the maximum cardinality of an intersect-empty set in $S$. If $S$ is discrete, $H(S)$ is also equal to the maximum number of vertices of an empty polygon in $S$.
\end{proposition}

As a corollary, we can derive a union bound for Helly numbers.

\begin{lemma}\label{thm:Helly-union-bound}
If $S, T \subseteq \R^d$, then $H(S \cup T) \leq H(S) + H(T)$.
\end{lemma}
\begin{proof}
Any set $X \subseteq S \cup T$ of $H(S) + H(T) + 1$ points either contains $H(S) + 1$ points of $S$ or $H(T) + 1$ points of $T$. By \cref{thm:empty-polytopes}, the set $X$ cannot be intersect-empty, and therefore any intersect-empty set in $S \cup T$ has at most $H(S) + H(T)$ vertices.
\end{proof}

We need one last lemma.

\begin{lemma}\label{thm:convex-intersection-helly}
If $C$ is convex, then $H(S \cap C) \leq H(S)$.
\end{lemma}
\begin{proof}
If the intersection of every $H(S)$ members of $\mathcal F$ contains a point in $S \cap C$, then the intersection of every $H(S)$ members of $\mathcal G := \{ F \cap C : F \in \mathcal F\}$ contains a point in $S$. It follows that $\bigcap \mathcal G$ contains a point in $S$, which implies that $\bigcap \mathcal F$ contains a point in $S \cap C$.
\end{proof}

Alternatively, one could note that every intersect-empty set in $S \cap C$ is also intersect-empty in $S$.

\begin{theorem}
If $Q$ is a periodic product set in $\R^d$, then $\h(Q,n) \leq 4^d \rho(Q)^2$ for every $n \in \N$.
\end{theorem}
\begin{proof}
Factor $Q$ as $A_1 \times \cdots \times A_d$, suppose the period of $Q$ is $p \in \R^d$, and let $B$ be the fundamental box of $Q$. We set
\[  c_i(a,b) = \begin{cases}
        \big\lvert [a_i,b_i] \cap A_i \big\rvert &\text{ if } b_i \geq a_i \\
        \big\lvert [a_i,b_i + p_i] \cap A_i\big\rvert &\text{ if } b_i < a_i
    \end{cases} \]
for each $a,b \in B \cap Q$ and $i \in [d]$ and set
\[  Q_{a,b} = \{(x,y) \in Q^2 : x_i \equiv a_i \smallmod p_i \text{ and } y_i \equiv b_i \smallmod p_i \text{ for every } i \in [d]\} \]
for each $a,b \in B \cap Q$. Moreover, we define
\[ f_{a,b}(x,y) = \prod_{i=1}^d \Big(m_i(Q)\big((y_i-x_i)-c_i(a,b)\big)/p_i + c_i(a,b)\!\Big), \]
which is logconcave (and therefore min-concave) for every pair $(a,b) \in B \cap Q$, and
\[  f(x,y) = \begin{cases}
        f_{a,b}(x,y) &\text{if } y \geq x \text{ and } (x,y) \in Q_{a,b}\\
        0 &\text{if } y \not\geq x.
    \end{cases} \]
The value $f(x,y)$ is exactly the number of points of $Q$ in $\bx(x,y)$ when $x,y \in Q$ and $y \geq x$. Since $Q_{a,b} \cap f^{-1}([n,\infty)) = Q_{a,b} \cap f^{-1}_{a,b}([n,\infty))$ and $f_{a,b}$ is min-concave, \cref{thm:convex-intersection-helly} provides the upper bound $H\big(Q_{a,b} \cap f^{-1}([n,\infty))\big) \leq H(Q_{a,b})$. Each set $Q_{a,b}$ is an affine image of $\Z^{2d}$, so
\[  H\big(Q^2 \cap f^{-1}([n,\infty))\big)
    = H\bigg( \bigcup_{a,b \in B\, \cap\, Q}\! \big(Q_{a,b} \cap f^{-1}([n,\infty))\big)\!\bigg)
    \leq\! \sum_{a,b \in B \cap Q}\!\! H(Q_{a,b})
    = \rho(Q)^2 2^{2d}\]
by \cref{thm:Helly-union-bound} and Doignon's theorem. A set contains a box with at least $n$ points in $Q$ if and only if it contains such a box that is parametrized in $Q \times Q$. Appealing to \cref{thm:basic-box-Helly} completes the proof.
\end{proof}

We quickly prove two more consequences of \cref{thm:box-weighting-min-concave}. In the continuous setting, there is no exact Helly-type theorem for surface area, although Rolnick and Sober\'on \cite{quant-p-q-theorems} proved a version with arbitrarily small loss. There is, however, such a theorem with boxes as witness sets.

Recall that the $k$-skeleton of a polytope is the union of its $k$-dimensional faces.

\begin{proposition}\label{thm:box-k-skeleton}
Let $\mathcal F$ be a finite family of convex sets in $\R^d$ and fix some $w \in \R$. If the intersection of any $2d$ or fewer members of $\mathcal F$ contains an axis-parallel box whose $k$-skeleton has $k$-dimensional volume at least $w$, then $\bigcap \mathcal F$ contains such a box, as well.
\end{proposition}
\begin{proof}
The $k$-dimensional volume of the $k$-skeleton of $\bx(x,y)$ is given by the elementary symmetric polynomial
\[ f(x,y) = 2^{d-k}\sum_{V \in \binom{[d]}{k}}\ \prod_{i \in V} (y_i - x_i). \]
Marcus and Lopes \cite{symmetric-functions-inequalities} proved that $f^{1/k}$ is concave on the set $\{ (x,y) \in \R^d \times \R^d : y \geq x \}$, which implies that $f$ is a min-concave box-weighting function. Applying \cref{thm:box-weighting-min-concave} with $S = T = \R^d$  finishes the proof.
\end{proof}

Another application of elementary symmetric polynomials yields the following extension of \cref{thm:box-helly-integer-lattice}.

\begin{proposition}\label{thm:k-box-discrete}
Let $\mathcal F$ be a finite family of convex sets in $\R^d$. If the intersection of any $2^{2d-1}$ or fewer members of $\mathcal F$ contains an axis-parallel box of dimension (at most) $k$ that itself contains at least $n$ lattice points, then $\bigcap \mathcal F$ contains a box of dimension (at most) $k$ with at least $n/\binom{d}{k}$ lattice points.
\end{proposition}
\begin{proof}
Apply \cref{thm:box-weighting-min-concave} to the elementary symmetric polynomial
\[  f(x,y) = \sum_{V \in \binom{[d]}{k}}\ \prod_{i \in V} (y_i - x_i + 1) \]
with $S = T = \Z^d$ and $w=n$ to obtain $a,b \in \Z^d$ such that $\bx(a,b) \subseteq \bigcap \mathcal F$ and $f(a,b) \geq n$. By the majority principle, there is a set $U \in \binom{[d]}{k}$ so that $\prod_{i \in U} (b_i - a_i + 1) \geq n/\binom{d}{k}$; setting $b' = \sum_{i \in U} b_ie_i + \sum_{i \notin U} a_ie_i$, the number of lattice points in the $k$-dimensional $\bx(a,b')$ is at least $n/\binom{d}{k}$.
\end{proof}

An essentially identical proof using the function in the proof of \cref{thm:box-k-skeleton} yields the following continuous version.

\begin{proposition}\label{thm:k-box-cont}
Let $\mathcal F$ be a finite family of convex sets in $\R^d$. If the intersection of any $2d$ or fewer members of $\mathcal F$ contains an axis-parallel box of dimension (at most) $k$ with $k$-dimensional volume at least $w$, then $\bigcap \mathcal F$ contains a box of dimension (at most) $k$ with $k$-volume at least $w/\binom{d}{k}$.
\end{proposition}

The collection of $2d$ half-spaces whose intersection is the unit cube show that the number $2d$ in propositions \ref{thm:box-k-skeleton} and \ref{thm:k-box-cont} cannot be reduced to $2d-1$. Doignon's theorem shows that $2^{2d-1}$ in \cref{thm:k-box-discrete} cannot be replaced by any number smaller than $2^d$. The optimal Helly number in this case is not clear.

\subsection{Fractional Helly-type theorems}

Although in this section we focus on the integer lattice, straightforward modifications to the proofs provide analogous statements for any periodic product set. A key ingredient in proving fractional versions of \cref{thm:box-helly-integer-lattice} will be the mixed-integer Helly theorem (first proved by Hoffman \cite{binding-constraints}), which says that $H(\R^a \times \Z^b) = (a+1)2^b$. We also use the following definition.

\begin{definition}
A box where every edge has length at least $t$ is called \textit{$t$-thick}.
\end{definition}

\begin{proposition}\label{thm:lower-Helly-number1}
Let $\mathcal F$ be a finite family of convex sets in $\R^d$. Let $m \in \{1,\dots,d\}$ and $t \geq 0$. If the intersection of every $(m+1)2^{2d-m-1}$ or fewer members of $\mathcal F$ contains a $t$-thick box with at least $n$ lattice points, then $\bigcap \mathcal F$ contains a $t$-thick box with at least $n/\big(1+\frac{1}{\lfloor t\rfloor + 1}\big)^m$ lattice points.
\end{proposition}
\begin{proof}
The function
\[  f(x,y) = \begin{cases}
        \prod\limits_{i=1}^d (y_i - x_i + 1) &\text{if } \min_i \{y_i - x_i\} \geq t \\
        \ 0 &\text{otherwise,}
    \end{cases} \]
is min-concave on the set of boxes. If $\bx(x,y)$ is $t$-thick and contains at least $n$ lattice points, then $f(x,y) \geq n$. Applying \cref{thm:box-weighting-min-concave} (and the mixed-integer Helly theorem) to $\mathcal F$ with $S = \Z^d$ and $T = \R^m \times \Z^{d-m}$ provides a point $(a,b) \in \Z^d \times (\R^m \times \Z^{d-m})$ with 
$\bx(a,b) \subseteq \bigcap \mathcal F$ and $f(a,b) \geq n$. Thus, $\bx(a,b)$ is $t$-thick and contains
\[  \prod_{i=1}^m\lfloor b_i - a_i + 1\rfloor\prod_{i=m+1}^{d} (b_i - a_i + 1)\]
lattice points. Using the substitution $y_i = b_i - a_i + 1$, the product is bounded below by
\[
    \min \left\{ \prod_{i=1}^m \lfloor y_i \rfloor\prod_{i=m+1}^d y_i\, :\, y_i \geq t+1 \text{ and } \prod_{i=1}^d y_i = n \right\}.
\]
Via the product condition $\prod_{i=1}^d y_i = n$, this is equal to 
\begin{equation}\label{eq:fractional-min-one}
    \min \left\{ n\prod_{i=1}^m \frac{\lfloor y_i \rfloor}{y_i}\, :\, y_i \geq t+1 \text{ and } \prod_{i=1}^d y_i = n \right\}.
\end{equation}
Since
\[
    \frac{\lfloor y_i\rfloor}{y_i}
    \geq \frac{\lfloor y_i\rfloor}{\lfloor y_i\rfloor + 1},
\]
the minimum \eqref{eq:fractional-min-one} is bounded below by the value of $n\prod_{i=1}^m \frac{\lfloor y_i\rfloor}{\lfloor y_i\rfloor + 1}$ when $y_i = t + 1$ for every $1 \leq i \leq d$, namely
\[
    \frac{n}{\left(1+\frac{1}{\lfloor t\rfloor + 1}\right)^m}.\qedhere
\]
\end{proof}

Dropping the thickness requirement on boxes (that is, ``requiring'' 0-thick boxes) guarantees $n/2^m$ points in $\bigcap \mathcal F$, a relatively small fraction of $n$. On the other hand, the proportion of points guaranteed by \cref{thm:lower-Helly-number1} increases rapidly as $t$ increases, which shows that the loss is an artifact of the possible thinness of the box in many directions; thickening the boxes guarantees a much larger fraction of points. In fact, taking $t = d-1$ in \cref{thm:lower-Helly-number1} guarantees a box with at least $n/e$ lattice points in the intersection.

The Helly number can be reduced even further in exchange for a smaller intersection.

\begin{proposition}\label{thm:lower-Helly-number2}
Let $\mathcal F$ be a finite family of convex sets in $\R^d$. Let $m \in \{1,\dots,d-1\}$ and $t \geq 1$. If the intersection of every $(d+m+1)2^{d - m - 1}$ or fewer members of $\mathcal F$ contains a $t$-thick box with at least $n$ lattice points, then $\bigcap \mathcal F$ contains a $t$-thick box with at least $n/\big((1 + \frac{1}{\lfloor t\rfloor+1})^{d-m}(1 + \frac{2}{\lfloor t\rfloor})^m\big)$ points. The same statement is true for $m=d$ if $(d+m+1)2^{d - m - 1}$ is replaced by $2d$.
\end{proposition}
\begin{proof}
The proof is similar to that of \cref{thm:lower-Helly-number1}. We apply \cref{thm:box-weighting-min-concave} to $\mathcal F$ with the same function $f$ and $S = \R^m \times \Z^{d-m}$ and $T = \R^d$ to get a point $(a,b) \in (\R^m \times \Z^{d-m}) \times \R^d$ such that $\bx(a,b) \subseteq \bigcap \mathcal F$ and $f(a,b) \geq n$. The set $\bx(a,b)$ contains at least
\[  \prod_{i=1}^m \lfloor b_i - a_i\rfloor \prod_{i=m+1}^{d} \lfloor b_i-a_i + 1\rfloor \]
lattice points. Substituting $y_i = b_i - a_i$, this product is bounded below by 
\[
    \min\left\{ \prod_{i=1}^m \lfloor y_i\rfloor\prod_{i=m+1}^{d} \lfloor y_i + 1\rfloor\,
        :\, y_i \geq t \text{ and } \prod_{i=1}^d (y_i+1) = n \right\}.
\]
Via the product condition $\prod_{i=1}^d (y_i+1) = n$, this is equal to
\begin{equation}\label{eq:fractional-min-two}
    \min\left\{ n\prod_{i=1}^m \frac{\lfloor y_i\rfloor}{y_i+1}\prod_{i=m+1}^{d} \frac{\lfloor y_i \rfloor + 1}{y_i+1}\,
        :\, y_i \geq t \text{ and } \prod_{i=1}^d (y_i+1) = n \right\}.
\end{equation}

Using the inequalities
\[
    \frac{\lfloor y_i\rfloor}{y_i+1}
    \geq \frac{\lfloor y_i\rfloor}{\lfloor y_i\rfloor + 2}
    \text{\quad and \quad}
    \frac{\lfloor y_i \rfloor + 1}{y_i+1}
    \geq \frac{\lfloor y_i\rfloor + 1}{\lfloor y_i\rfloor + 2},
\]
the minimum \eqref{eq:fractional-min-two} is bounded below by the value of $n\prod_{i=1}^m \frac{\lfloor y_i\rfloor}{\lfloor y_i\rfloor+2}\prod_{i=m+1}^{d} \frac{\lfloor y_i \rfloor + 1}{\lfloor y_i\rfloor +2}$ when $y_i = t$ for every $1 \leq i \leq d$, namely
\[
    \frac{n}{%
        \left(1+\frac{1}{\lfloor t\rfloor+1}\right)^{d-m}%
        \left(1+\frac{2}{\lfloor t\rfloor}\right)^m}.\qedhere
\]
\end{proof}

As before, if the boxes are thick enough, the intersection contains a box with a constant proportion of lattice points; taking $t = d$ guarantees $n/e^2$ lattice points in the intersection, for example.

\cref{thm:lower-Helly-number} follows as a special case.

\begin{proof}[Proof of \cref{thm:lower-Helly-number}]
Apply \cref{thm:lower-Helly-number2} with $m=d$ and $t=1$.
\end{proof}

No theorem of this type has a weaker intersection condition.

\begin{example}\label{ex:impossible-fractional}
Let $t \geq 1$ and $\mathcal F$ be the collection of $2d$ half-spaces whose intersection is the hypercube $[0,t]^d$. The intersection of any $2d-1$ members of $\mathcal F$ contains $t$-thick boxes with arbitrarily many lattice points, but $\bigcap \mathcal F$ contains only $(t+1)^d$ lattice points. So a positive proportion of lattice points cannot be guaranteed in $\bigcap \mathcal F$. (Similar hypercube examples have been used for quantitative Helly-type theorems since B\'ar\'any, Katchalski, and Pach's paper \cite{BKP82} in 1982.)
\end{example}

The following example shows that an exponential loss is to some extent unavoidable in results like propositions \ref{thm:lower-Helly-number1} and \ref{thm:lower-Helly-number2}, even for 1-thick boxes.

\begin{example}\label{ex:exponential-loss-lower-Helly-number}
Let $\mathcal F$ be the collection of $2^d$ half-spaces whose intersection is the cross-polytope $\conv(\pm e_1,\dots, \pm e_d)$. The intersection of any $2^d-1$ of these half-spaces contains a 1-thick box with $2^d$ lattice points, but any box in the cross-polytope contains at most 3 lattice points.
\end{example}

\begin{question}
Is the condition of 1-thickness necessary in \cref{thm:lower-Helly-number} and \cref{thm:lower-Helly-number2}?
\end{question}

In \cite{Helly-algebraic-subsets}, De Loera, La Haye, Oliveros, and Rold\'an-Pensado provided a set of sufficient conditions that guarantees the colorful version of a given Helly-type theorem by abstracting Lov\'asz's original proof. Although not stated in their paper, a direct application of their Theorem 5.3 yields the following result.

\begin{theorem}[Colorful mixed-integer Helly theorem]\label{thm:colorful-mixed-integer}
Set $h = (a+1)2^b$ and let $\mathcal F_1, \dots,\mathcal F_h$ be nonempty finite families of convex sets in $\R^{a+b}$. If $\bigcap_{i=1}^h F_i$ contains a point in $\R^a \times \Z^b$ for every choice of sets $F_i \in \mathcal F_i$, then there is an index $k \in [h]$ so that $\bigcap \mathcal F_k$ contains a point in $\R^a \times \Z^b$.
\end{theorem}

The corresponding colorful versions of \cref{thm:lower-Helly-number1} and \cref{thm:lower-Helly-number2} are also true. We stated the vanilla-flavor Helly-type theorems for simplicity, but the colorful incarnations can be proved by replacing the mixed-integer Helly theorem by \cref{thm:colorful-mixed-integer} in the proofs.

Traditionally, fractional Helly-type theorems guarantee that a large subfamily of $\mathcal F$ has a given intersection property. Along these lines, B\'ar\'any and Matou\v{s}ek \cite{Barany-Matousek-fractional-integer} proved the surprising fact that a fractional Helly-type theorem holds for $\Z^d$ with a local intersection condition on $(d+1)$-tuples, despite the Helly number for $\Z^d$ being $2^d$.

\begin{theorem}[B\'ar\'any-Matou\v{s}ek, 2003]\label{thm:barany-matousek}
For every $d\geq 1$ and $\alpha \in (0,1]$ there exists a $\beta > 0$ such that the following is true. If $\bigcap \mathcal H$ contains a lattice point for at least $\alpha \binom{\lvert \mathcal F\rvert}{d+1}$ subcollections $\mathcal H \subseteq \mathcal F$ of size $d+1$, then there is a subfamily $\mathcal G \subseteq \mathcal F$ with at least $\beta \lvert \mathcal F \rvert$ sets whose intersection contains a lattice point.
\end{theorem}

The B\'ar\'any-Matou\v{s}ek theorem produces a fractional theorem that complements propositions \ref{thm:lower-Helly-number1} and \ref{thm:lower-Helly-number2}.

\begin{corollary}
For every $d\geq 1$ and $\alpha \in (0,1]$, there exists a $\beta > 0$ such that the following is true. If $\bigcap \mathcal H$ contains a box with at least $n$ lattice points for at least $\alpha \binom{\lvert \mathcal F\rvert}{2d}$ subfamilies $\mathcal H \subseteq \mathcal F$ of size $2d$, then there is a collection $\mathcal G \subseteq \mathcal F$ with at least $\beta \lvert \mathcal F\rvert$ sets whose intersection contains a box with at least $n$ lattice points.
\end{corollary}

The proof mimics that of \cref{thm:colorful-box-helly-integer}, substituting \cref{thm:barany-matousek} for the colorful Helly theorem in the application of  \cref{thm:box-weighting-min-concave}.

\section{Integer sets with infinite Helly number}

We begin with a lower bound for the Helly number of product sets in the plane.

\begin{proposition}\label{thm:product-sets-lower-bound}
Given a discrete set $A \subseteq \R$, list its elements in increasing order as $\cdots < a_{-1} < a_0 < a_1 < \cdots$ and define
\[ t(n) = \frac{a_{n+2} - a_{n+1}}{a_{n+1} - a_n}. \]
If there exist $b \in \Z$ and $m \in \N$ so that $t(n) \geq t(n+1)$ for every $b \leq n \leq b+m-1$ and strict inequality holds for $k$ values of $n$, then $H(A^2) \geq k+4$.
\end{proposition}
\begin{proof}
We define
\begin{equation}\label{eq:3-1}
    P = \conv\!\big( \{(a_b,a_b),(a_{b+m+2},a_{b+m+2})\} \cup \{(a_i,a_{i+1}) : b \leq i \leq b+m+1\}\big).
\end{equation}
The polygon $P$ contains no other points of $A^2$ in its interior. The condition that $t(n) \geq t(n+1)$ guarantees that the slopes of the lines connecting consecutive points of the form $(a_i,a_{i+1})$ are nonincreasing, so the points defining $P$ in \eqref{eq:3-1} are all on the boundary of $P$. If $t(n) > t(n+1)$, then $(a_{n+1},a_{n+2})$ is a vertex of $P$. The condition on strict inequality guarantees that $P$ has $k+4$ vertices (the $k$ where strict inequality holds, plus the points $(a_b,a_b),\, (a_b,a_{b+1}),\, (a_{b+m},a_{b+m+1}),$ and $(a_{b+m+1},a_{b+m+1})$). By \cref{thm:empty-polytopes}, $H(A^2) \geq k+4$.
\end{proof}

A computer search of the values of $t(n)$ for the prime lattice in the interval $[0,10^6]$ shows that $H(\mathcal P \times \mathcal P) \geq 13$.\footnote{The empty 13-gon on the main diagonal of $\mathcal P \times \mathcal P$ whose vertex $x$-coordinates are the consecutive primes 258500509, 258500527, 258500549, 258500569, 258500587, 258500603, 258500617, 258500629, 258500639, 258500647, 258500651, and 258500659 is convex.} 

A Helly-type theorem for $A^d$ in $\R^d$ implies a corresponding Helly-type theorem for $A^k$ in $\R^k$ for all $1 \leq k \leq d$. Consequently, the lack of such a theorem in the plane implies the lack of one in all higher dimensions.

\begin{corollary}\label{thm:polynomial-infinite-Helly}
If $p$ is a polynomial with degree at least 2 and $A = \{ p(n) : n \in \Z\}$, then $H(A^d) = \infty$ for all $d \geq 2$.
\end{corollary}
\begin{proof}
It suffices to prove that $H(A^2) = \infty$. We may assume that $p$ is monic, since multiplying $p$ by a nonzero scalar is an affine transformation of $A^2$ and preserves the Helly number. The polynomial $p(n)$ is eventually increasing, which means that we may shift the indices so that $a_n = p(n)$ for large enough $n$. Substituting $p(n) = n^k + an^{k-1} + O(n^{k-2})$ in the expression for $t(n)$ yields
\[  t(n)
    = 1 + \frac{p(n+2) - p(n+1) +p(n)}{p(n+1) - p(n)}
    = 1 + \frac{O(n^{k-2})}{kn^{k-1} + O(n^{k-2})}. \]
It follows that $t(n)$ is eventually strictly decreasing, and \cref{thm:product-sets-lower-bound} implies that $H(A^2)$ is not finite.
\end{proof}

\begin{corollary}
Let $A_k = \{n^k : n \in \Z\}$ and $B_k = \{ \binom{n}{k} : n \in \N\}$. If $d,k \geq 2$, then $h(A_k^d) = h(B_k^d) = \infty$.
\end{corollary}

Unfortunately, there are some sets for which \cref{thm:product-sets-lower-bound} gives no information whatsoever. For example, if $A = \{2^n : n \in \N\}$, then $t(n) = 2$ for every $n$.

\begin{question}
What is $H(\{2^n : n \in \N\}^2)$?
\end{question}

Given a set $A \subseteq \Z$, let $\delta(A) = \liminf_{n\to\infty} \big\lvert A \cap [-n,n]\big\rvert / (2n+1)$ be its \textit{lower density} in $\Z$. It seems intuitive that results in \cref{thm:polynomial-infinite-Helly} are attributable to the sparseness of $A$; that is, that $\delta(A) = 0$. We can easily construct a set $A \subseteq \Z$ with positive density such that $H(A^2) = \infty$ using \cref{thm:polynomial-infinite-Helly}: Simply intersperse intervals of $\Z$ with well-chosen intervals of $\{n^3 : n \in \N\}$ to create a set $A$ so that $A^2$ contains arbitrarily large empty polygons, but (if the intervals of $\Z$ are long enough) has positive lower density.

The resulting set has infinite Helly number because it is not dense ``everywhere''---there are still arbitrarily large gaps between consecutive elements. A set $A \subseteq \N$ is called \textit{$\ell$-syndetic} if the differences between consecutive elements of $A$ are bounded above by $\ell$. If $A$ is syndetic, then it is, in some sense, dense everywhere. It therefore seems plausible that $H(A^2)$ is finite when $A$ is a syndetic set, but in this we are foiled yet again: Surprisingly, bounded gaps are still insufficient to guarantee a finite Helly number.

To prove \cref{thm:syndetic-set-infinite-helly}, we will use Dirichlet's theorem on rational approximation.

\begin{theorem}[Dirichlet \cite{diophantine-approx}, 1842]
If $\alpha \in \R$ is irrational, then there are infinitely many coprime integers $p$ and $q$ so that $\lvert q\alpha - p\rvert \leq 1/q$.
\end{theorem}

\begin{proof}[Proof of \cref{thm:syndetic-set-infinite-helly}]
The basic outline is that we cheat. We first construct a set of polygons, and then we build up a syndetic point set around them to guarantee they are empty, ensuring that the Helly number is infinite.

Fix an irrational number $\alpha > 2$. Let $L$ be the line through the origin with slope $\alpha$ and $M$ be the strip $L + \big(\{0\} \times [0,1)\big)$. (Here, $+$ denotes the Minkowski sum.) We first construct an infinite set of integer polygons.

Let $(q_i,p_i)_{i=1}^\infty$ be a sequence of pairs of positive integers so that $0 < \lvert p_i - q_i\alpha\rvert < 1/q_i$. The quantity $p_i - q_i\alpha$ is either positive for infinitely many values of $i$ or negative for infinitely many values of $i$. Since the arguments are symmetric, we suppose that the former is the case and choose a subsequence so that $p_i -q_i\alpha > 0$ for every $i \in \N$. The fractions $p_i/q_i$ tend toward $\alpha$, so we may take a further subsequence such that $p_i/q_i > p_{i+1}/q_{i+1}$ for every $i \in \N$.

For each $n \in \N$, choose some index $j$ so that $q_j > n$, set $v^n_k = \sum_{i=j}^{j+k} (q_i,p_i)$ for every integer $0 \leq i \leq n-1$, and define $P_n = \conv(0,v^n_0, \dots, v^n_{n-1})$. The condition on decreasing quotients guarantees that $P_n$ has exactly $n+1$ vertices, namely $\{0,v^n_0,\dots,v^n_{n-1}\}$. Moreover, since $0 < p_i - q_i\alpha < 1/q_i < 1/n$, every vector $v^n_i$ is contained in $M$, showing that $P_n \subseteq M$.

Let $\pi_1,\pi_2\colon \R^2 \to \R$ be the orthogonal projections onto the $x$- and $y$-axes, respectively. We translate each polygon $P_n$ by an integer vector $w_n$ such that any two projections in $\{\pi_1(P_n+w_n) : n \in \N\} \cup \{\pi_2(P_n+w_n) : n \in \N\}$ are disjoint; let $Q_n = P_n + w_n$. Defining $B = \bigcup_{n=1}^\infty \big(\pi_1(\operatorname{vert}(Q_n)) \cup \pi_2(\operatorname{vert}(Q_n))\big)$, each $Q_n$ is an empty polygon in $B^2$.

To finish the proof, we construct a 2-syndetic set $A$ containing $B$ so that each $Q_n$ remains empty in $A^2$. For each $m \in \Z$, exactly one of the following holds:
\begin{enumerate}[label=(\arabic*),nolistsep]
    \item $m \in \pi_1(Q_n)$ for exactly one $n$,
    \item $m \in \pi_2(Q_n)$ for exactly one $n$, or
    \item $m \in \N \setminus \bigcup_{n=1}^\infty \big(\pi_1(Q_n) \cup \pi_2(Q_n)\big)$.
\end{enumerate}
Let us denote the set of integers that satisfy (1), (2), and (3) by $\Pi_1$, $\Pi_2$, and $\Pi_3$, respectively. Finally, we set $M_n = (M + w_n) \cap \big(\pi_1(Q_n) \times \pi_2(Q_n)\big)$, so $Q_n \subseteq M_n$.

The vertical width of $M_n$ is 1 and the slope of the strip is an irrational number strictly greater than 2. This means that if $m \in \pi_2(M_n \cap \Z^2)$, then $m+1 \notin \pi_2(M_n \cap \Z^2)$. The set
\[ A := \Pi_1 \cup \Big(\Pi_2 \setminus \big(\bigcup_{n=1}^\infty \pi_2(M_n \cap \Z^2)\big)\Big) \cup \Pi_3 \cup B \]
is therefore 2-syndetic. Since $M_n \cap A^2 = M_n\cap B^2$, the polygons $Q_n$ are empty in $A^2$. Thus $A^2$ contains arbitrarily large empty polygons, so $H(A^2) = \infty$ by \cref{thm:empty-polytopes}.
\end{proof}
\vspace{3ex}

\noindent
{\large \textbf{Acknowledgments}}\\
This research project was completed as part of the 2020 Baruch Discrete Mathematics REU, supported by NSF awards DMS-1802059, DMS-1851420, and DMS-1953141. The author thanks Pablo Sober\'on especially for his guidance and support throughout this project.
\vspace{3ex}

\bibliography{bibliography.bib}
\bibliographystyle{amsplain-nodash}
\end{document}